\documentclass[11pt]{amsart}

\usepackage{amsfonts}
\usepackage{amssymb, latexsym, amsthm, amsmath, verbatim}
\usepackage{indentfirst}

\numberwithin{equation}{section}
\theoremstyle{definition}
\newtheorem{Thm}[equation]{Theorem}
\newtheorem{Prop}[equation]{Proposition}
\newtheorem{Cor}[equation]{Corollary}
\newtheorem{Lem}[equation]{Lemma}

\newtheorem{Exa}[equation]{Example}
\newtheorem{Rmk}[equation]{Remark}
\newtheorem{Que}{Question}

\makeatletter
\def\imod#1{\allowbreak\mkern5mu{\operator@font mod}\,\,#1}
\makeatother

\begin{document}

\title[Eigenvalue of Fricke Involution]{Eigenvalue of Fricke Involution on Newforms of Level $4$ and of Trivial Character}
\author[Yichao Zhang]{Yichao Zhang}
\address{Department of Mathematics, University of Connecticut, Storrs, CT 06269}
\email{yichao.zhang@uconn.edu}
\date{}

\begin{abstract}
In this note, we consider the newforms of integral weight, level $4$ and of trivial character, and prove that all of them are actually level $1$ forms of some non-Dirichlet character. As a byproduct, we can prove that all of them are eigenfunctions of the Fricke involution with eigenvalue $-1$.
\end{abstract}

\maketitle

\section*{Introduction}
\noindent
The Fricke involution $W_N$ of level $N$, also known as the canonical involution, acts on the space of newforms of level $N$, some integral weight $k$, and trivial character. Here $k$ is necessarily even and positive. It is well-known that Hecke eigenforms behave well under the Fricke involution. More specifically, if $f$ is a normalized Hecke eigenform of some level $N$, weight $k$ and of trivial character, then we have $f|_kW_N=cg$ with $c\in\mathbb C^\times$ and $g$ another normalized Hecke eigenform in the same space. The Fourier coefficients of $g$ can be explicitly determined by that of $f$ but the scalar $c$ is left mysterious. 
\begin{Que} can we explicitly determine $c$ with information of $f$?
\end{Que}
\noindent
If $N$ is square-free, one can express $c$ in terms of $a_p$ with $p\mid N$; for example, one may work it out explicitly using Corollary 4.6.18 in \cite{miyake2006modular} and a similar result as Lemma 2.1 in \cite{zhang2013vector}. For example, if $N=p$ is a prime and assume $f=\sum_na_nq^n$, then $g=f$ and $c=-p^{1-\frac{k}{2}}a_p\in\{\pm 1\}$. The determination of $c$ in the case of non-square-free $N$ is more subtle, due the vanishing of some Fourier coefficients $a_p$ when $p^2\mid N$. In this note, we consider this question in the simplest case, namely the case when $N=4$. 

On the other hand, the structure of $\text{SL}_2(\mathbb Z)$ has been known for a long time and it is generated by the matrices $S,T$ up to some defining relations. Explicitly, for the modular group $\text{PSL}_2(\mathbb Z)=\text{SL}_2(\mathbb Z)/\{\pm I\}$, we have
\[\text{PSL}_2(\mathbb Z)=\langle S,T: S^2=(ST)^3=I\rangle,\]
where $S,T$ denote their corresponding images in $\text{PSL}_2(\mathbb Z)$. We can use instead the pair of generators $S$ and $ST$, and it follows that $\text{PSL}_2(\mathbb Z)$, hence $\text{SL}_2(\mathbb Z)$, may possess some non-trivial characters $\chi$. For example, Conrad \cite{conradsl2} gave an explicit example of such a $\chi$ of (maximal) order $12$ in his online notes on $\text{SL}_2(\mathbb Z)$. If we restrict our attention to the modular group $\text{PSL}_2(\mathbb Z)$, such a $\chi$ has order at most $6$. Note that since these characters are level $1$ characters, they are not induced by Dirichlet characters on the level. A natural question then arises:
\begin{Que} can the space $\mathcal S(\text{SL}_2(\mathbb Z),k,\chi)$ be non-zero for some weight $k$ and some non-trivial character $\chi$?
\end{Que}
\noindent
In other words, are there non-zero cusp forms of level $1$, type $\chi$ and of some weight $k$? The answer is positive and we shall consider a real non-trivial character $\chi$.

It turns out that these two questions are related. They can be both answered by an isomorphism between the space of newforms of level $4$ and of trivial character and that of modular forms of level $1$ and type $\chi$. Here $\chi$ is the only non-trivial real character of $\text{PSL}_2(\mathbb Z)$ and it satisfies $\chi(S)=\chi(T)=-1$. Such an isomorphism is given in Theorem 1.2 below. In particular, Theorem 1.2 says that the level $4$ newforms are actually level $1$ forms, that is, they are ``oldforms" when non-Dirichlet characters are taken into account. This is analogous to the results in \cite{zhang2013vector}, where we proved that forms in a subspace of the space of modular forms of level $N$ and character $\left(\frac{N}{\cdot}\right)$ are actually level $1$ (vector valued) modular forms for some $N$; for example, $N\equiv 1\imod 4$ and $N>1$.

Such an isomorphism answers Question 2 directly and the first such example appears when $k=6$. It also answers Question 1, since under the isomorphism, the eigenvalue of the Fricke involution is given by $\chi(S)$ which is $-1$. Actually, we always have $\chi(S)=\chi(T)$ when $\chi$ is real. This is Theorem 1.3, with which we may make the functional equation of the L-function associated to $f$ more concrete.

In Section 1, we set up the notations and state the main results. In Section 2, we prove some results of modular forms between different levels (or rather, groups), and also consider possible real characters of $\text{PSL}_2(\mathbb Z)$. In the last section, we give the proof of Theorem 1.2 and Theorem 1.3 and end the note with an example.

\section{Statements of Main Theorems}

\noindent
In this section, we fix the notations and state the main results. For unexplained notations and terminology, we refer the readers to any standard textbook on modular forms, for example \cite{miyake2006modular}.

For a positive integer $N$, we have congruence subgroups $\Gamma_0(N)$ and $\Gamma(N)$ defined as follows:
\[\Gamma_0(N)=\left\{\begin{pmatrix}
a&b\\
c&d
\end{pmatrix}\in \text{SL}_2(\mathbb Z): \begin{pmatrix}
a&b\\
c&d
\end{pmatrix}\equiv \begin{pmatrix}
*&*\\
0&*
\end{pmatrix}\imod N\right\},\]
\[\Gamma(N)=\left\{\begin{pmatrix}
a&b\\
c&d
\end{pmatrix}\in \text{SL}_2(\mathbb Z): \begin{pmatrix}
a&b\\
c&d
\end{pmatrix}\equiv \begin{pmatrix}
1&0\\
0&1
\end{pmatrix}\imod N\right\},\]
where $*$ means no restriction on the corresponding entry. We shall be only interested in the case  when $N=2$ or $4$.

Let $k$ be an even positive integer and $\mathbb H$ the upper half plane. Recall that for a real matrix $M$ of positive determinant and a function $f$ on $\mathbb H$, the weight-$k$ slash operator of $M$ is defined by
\[(f|_kM)(\tau)=\text{det}(M)^{\frac{k}{2}}(c\tau+d)^{-k}f(M\tau),\quad M=
\begin{pmatrix}
a&b\\
c&d
\end{pmatrix},\]
where $M\tau=\frac{a\tau+b}{c\tau+d}$.
Let $\Gamma$ be any congruence subgroup of $\text{SL}_2(\mathbb Z)$ and $\chi$ be any character (homomorphisms into $\mathbb C^\times$ of finite order) of $\Gamma$. Denote by $\mathcal S(\Gamma,k,\chi)$ the space of cusp forms for $\Gamma$ of weight $k$ and of character $\chi$, namely holomorphic functions $f$ on $\mathbb H$ such that $f|_kM=\chi(M)f$ for all $M\in\Gamma$ and $f$ vanishes at cusps. If $\Gamma=\Gamma_0(N)$ and $\chi$ is a Dirichlet character modulo $N$, we denote $\mathcal S^{\text{new}}(\Gamma_0(N),k,\chi)$ the subspace of newforms for $\Gamma_0(N)$ of weight $k$ and of character $\chi$. 
Recall that the Fricke involution, defined by the weight $k$ slash operator of 
\[W_N=\begin{pmatrix}
0&-1\\
N&0
\end{pmatrix},\]
acts on the space $\mathcal S^{\text{new}}(\Gamma_0(N),k,1)$. We have also the Hecke operators $T_p$, for $p\nmid N$ a rational prime, and $U_p$ for $p\mid N$, and we write down the action of $U_p$ explicitly as follows
\[f|_kU_p=p^{\frac{k}{2}-1}\sum_{j\imod p}f\left|_k\begin{pmatrix}
1&j\\
0&p
\end{pmatrix}\right..\] Note that there are different normalizations in the literature on the Hecke operators.
For ease of notations, we denote
\[I=\begin{pmatrix}
1&0\\
0&1
\end{pmatrix},\quad S=\begin{pmatrix}
0&-1\\
1&0
\end{pmatrix},\quad T=\begin{pmatrix}
1&1\\
0&1
\end{pmatrix},\quad T_{1/2}=\begin{pmatrix}
1&\frac{1}{2}\\
0&1
\end{pmatrix},\quad V_N=\begin{pmatrix}
N&0\\
0&1
\end{pmatrix}.\]

Assume that $f$ has Fourier expansion $\sum_na_nq^n$ at $\infty$; here $q=e^{2\pi i\tau}$. Recall that a non-zero $f$ is called a Hecke eigenform or a primitive form if $f$ is a common eigenfunction for all the Hecke operators $T_p$. If so, $a_1\neq 0$ and $f$ is called normalized if $a_1=1$. We recall the following well-known result:

\begin{Lem}
Let $f$ be a normalized Hecke eigenform in $\mathcal S^{\text{new}}(\Gamma_0(N),k,1)$ and let $f=\sum_na_nq^n$ be its Fourier expansion at $\infty$. Then

(1) $f|_kT_p=a_pf$ for each $p\nmid N$.

(2) $f$ is an eigenfunction for $U_p$ and $f|_kU_p=a_pf$ for each $p\mid N$.

(3) $f|_kW_N=cg$, where $c\in\mathbb C^\times$, $g=\sum_nb_nq^n$ is a normalized Hecke eigenform such that $b_p=a_p$ if $p\nmid N$ and $b_p=\overline{a_p}$ if $p\mid N$.
\end{Lem} 

This lemma is part of the theory of newforms, also known as Atkin-Lehner-Li theory (\cite{atkin1970hecke} and \cite{li1975newforms}). For proofs of this lemma, one may see Chapter 4 of \cite{miyake2006modular}. Let $\chi$ be the non-trivial real character on $\text{PSL}_2(\mathbb Z)$ such that $\chi(S)=\chi(T)=-1$. Actually this is the unique one (see Lemma 2.1 below).

Now we state our main theorems.
\begin{Thm}
The map
\begin{eqnarray*}
\mathcal S(\text{SL}_2(\mathbb Z),k,\chi)&\rightarrow&\mathcal S^{\text{new}}(\Gamma_0(4),k,1)\\
f(\tau)&\mapsto& g(\tau)=f(2\tau)
\end{eqnarray*}
defines an isomorphism.
\end{Thm}

As a byproduct, we obtain
\begin{Thm}
For any $g\in \mathcal S^{\text{new}}(\Gamma_0(4),k,1)$, we have $g|_kW_4=-g$.
\end{Thm}

For any $g=\sum_na_nq^n\in \mathcal S^{\text{new}}(\Gamma_0(4),k,1)$, we have its L-function 
\[L(s,g)=\sum_na_nn^{-s},\] and its completed L-function $\Lambda(s,g)=\pi^{-s}\Gamma(s)L(s,g)$. It is well-known that $\Lambda(s,g)$ can be analytically continued to the whole $s$-plane and satisfies the following functional equation \[\Lambda(s,g)=i^k\Lambda(k-s,g|_kW_4).\] 
With Theorem 1.3, this functional equation can be made more precise:
\begin{Cor}
For any $g\in \mathcal S^{\text{new}}(\Gamma_0(4),k,1)$, we have $\Lambda(s,g)=-i^k\Lambda(k-s,g)$.
\end{Cor}

\begin{Rmk}
We content ourselves in this note with the case when $N=4$ and $\chi$ is real, however, similar results are expected to hold in a general setting. For example, there should be similar relations between the spaces $\mathcal S(\Gamma_0(N),k,\chi)$ and $\mathcal S^{\text{new}}(\Gamma_0(4N),k,1)$ for a square-free odd $N$ and some non-Dirichlet character $\chi$.
\end{Rmk}

\section{Modular Forms between Different Levels}
\noindent
As before, we denote also by $S,T$ their images in $\text{PSL}_2(\mathbb Z)$ respectively. It is well-known that $\text{PSL}_2(\mathbb Z)$ is generated by $S,T$; more precisely,
\[\text{PSL}_2(\mathbb Z)=\langle S,T: S^2=(ST)^3=I\rangle.\]

We consider characters $\chi$ of $\text{PSL}_2(\mathbb Z)$; that is, homomorphisms $\chi:\text{PSL}_2(\mathbb Z)\rightarrow \mathbb C^\times$ with finite images. In case of real characters, we have the following elementary lemma:
\begin{Lem}
Let $\chi$ be a character of $\text{PSL}_2(\mathbb Z)$. Then $\chi$ is a real character if and only if $\chi(T)$ is real, in which case either $\chi$ is trivial or $\chi$ is the unique character such that $\chi(S)=\chi(T)=-1$.
\end{Lem}
\begin{proof}
It is clear that $\chi(S)\in\{\pm 1\}$ and $\chi(T)$ is a sixth root of unity since $S^2=(ST)^3=I$, and the first statement follows. 

If $\chi$ is real, then $\chi(S),\chi(T)\in \{\pm 1\}$. Because of the defining relation $(ST)^3=I$, we must have $\chi(S)=\chi(T)$. Now that $S^2=(ST)^3=I$ are the only definition relations, the conditions $\chi(S)=\chi(T)=-1$ indeed give a character.
\end{proof}

\begin{Prop}
Let $\chi$ be a real character of $\text{PSL}_2(\mathbb Z)$. We have $f\in \mathcal S(\text{SL}_2(\mathbb Z),k,\chi)$ if and only if $f\in \mathcal S(\Gamma(2),k,1)$ and $f|_kS=\chi(S)f$, $f|_kT=\chi(T)f$.
\end{Prop}
\begin{proof} 
We first recall the well-known fact that $\Gamma_0(4)$, considered in $\text{PSL}_2(\mathbb Z)$, is generated by $T^2$ and $ST^4S$ (for a proof, see page 26 in \cite{zagier2008elliptic}). Since $\Gamma(2)$ and $\Gamma_0(4)$ are conjugate via $V_2$, we see that $\Gamma(2)$ is generated by $T^2$ and $ST^2S$. So clearly $\chi(\Gamma(2))=\{1\}$, and the proposition follows.
\end{proof}

\begin{Cor}
Let $\chi$ be a real character of $\text{PSL}_2(\mathbb Z)$. We have $f\in \mathcal S(\text{SL}_2(\mathbb Z),k,\chi)$ if and only if $g(\tau)=f(2\tau)\in \mathcal S(\Gamma_0(4),k,1)$ and $g|_kW_4=\chi(S)f$, $g|_kT_{1/2}=\chi(T)g$.
\end{Cor}
\begin{proof}
Since $\Gamma(2)$ and $\Gamma_0(4)$ are conjugate, $f(\tau)\mapsto f(2\tau)$ defines an isomorphism between the spaces of cusp forms on $\Gamma(2)$ and on $\Gamma_0(4)$. Under such conjugation, $S$ corresponds to $W_4$ and $T$ corresponds to $T_{1/2}$, up to $\pm I$, so the corollary follows from the previous proposition.
\end{proof}

\begin{Rmk}
If one wants to consider non-real characters $\chi$, he has to enlarge the level, since $\chi$ is no longer trivial on $\Gamma(2)$. 
\end{Rmk}

\section{Proof of Theorem 1.2 and 1.3}

\noindent
We begin with the following lemma.

\begin{Lem}
If $g\in\mathcal S^{\text{new}}(\Gamma_0(4),k,1)$, then $g|_kU_2=0$. Consequently, $g|_kT_{1/2}=-g$, and if $g$ is a Hecke eigenform then $g|_kW_4=cg$ for $c\in\{\pm 1\}$.
\end{Lem}
\begin{proof}
Since $g\in\mathcal S^{\text{new}}(\Gamma_0(4),k,1)$ contains a basis of Hecke eigenforms (see Theorem 4.6.13 of \cite{miyake2006modular}), we may assume that $f$ is a Hecke eigenform. We then have $g|_kU_2=a_2g$, and we only need to prove that $a_2=0$, which in turn follows from Theorem 4.6.17 in \cite{miyake2006modular}.

Note that $g|_kU_2=0$ implies that only odd powers in $q$ can appear in the Fourier expansion of $g$ at $\infty$, and we must have $g(\tau+\frac{1}{2})=-g(\tau)$. The last statement follows from  Theorem 4.6.16 in \cite{miyake2006modular}, since $a_2=0$ and hence real.
\end{proof}

\noindent
\emph{Proof of Theorem 1.3.} By Lemma 2.1 and Lemma 3.1, the involution $W_4$ on $\mathcal S^{\text{new}}(\Gamma_0(4),k,1)$ is diagonalizable with eigenvalues being either $+1$ or $-1$. Therefore we may decompose the space into the direct sum of the plus space and the minus space according to the sign of the eigenvalues. 

Suppose the plus space is not zero, and let $g$ be one of the non-zero forms therein, that is, $g|_kW_4=g$. Now consider $f(\tau)=g(\frac{\tau}{2})$, and we know that $f\in \mathcal S(\Gamma(2),k,1)$. From the transformation behavior of $g$ (see Lemma 3.1) and the isomorphism between $\Gamma_0(4)$ and $\Gamma(2)$, we see that
\[f|_kS=f,\quad \text{and}\quad f|_kT=-f.\]
Since the slash operator gives an action, we have $f|_kM=\chi(M)f$ with $\chi(M)\in\{\pm 1\}$
for each $M\in\text{PSL}_2(\mathbb Z)$ and $\chi$ is a homomorphism, hence a real character of $\text{PSL}_2(\mathbb Z)$. By Lemma 2.1, we see that this is not possible. This completes the proof of Theorem 1.3. $\Box$

\medskip

From now on, let $\chi$ be the unique character of $\text{PSL}_2(\mathbb Z)$ such that $\chi(S)=\chi(T)=-1$.
With little more effort, we may give a proof of Theorem 1.2.
\medskip

\noindent
\emph{Proof of Theorem 1.2.} From Theorem 1.3, we know that any $g\in \mathcal S^{\text{new}}(\Gamma_0(4),k,1)$ is an eigenfunction of both $T_{1/2}$ and $W_4$, with eigenvalues both equal to $-1$. By Corollary 2.3, this means that $f(\tau)=g(\frac{\tau}{2})\in \mathcal S(\text{SL}_2(\mathbb Z),k,\chi)$. 

Clearly, if the maps $f(\tau)\mapsto f(2\tau)$ and $g(\tau)\mapsto g(\frac{\tau}{2})$ are well-defined, they define inverse isomorphisms. So to complete the proof, we only have to verify that if $f\in \mathcal S(\text{SL}_2(\mathbb Z),k,\chi)$, then $g\in\mathcal S^{\text{new}}(\Gamma_0(4),k,1)$, that is, $g$ is in the space of newforms. Suppose not, and we must have that $g$ itself, without scaling on $\tau$, is a modular form of level $2$, since the only possible divisors of $4$ are $2$-powers and $g$ must only contain odd powers in $q$. That said, we have $g|_kST^2=g|_kS$, since $ST^2S\in\Gamma_0(2)$. Now $g|_kW_4=-g$ implies that $g|_kS=-g|_kV_4^{-1}$, and hence $g|_kV_4^{-1}T^2V_4=g$. Now we see that
\[g|_kV_4^{-1}T^2V_4=g|_kT_{1/2}=g,\]
contradicting to Corollary 2.3 since $\chi(T)=-1$.  Theorem 1.2 then follows. $\Box$

\medskip

We end this note with the following example.
\begin{Exa}
We set $k=6$ and we know that the space $\mathcal S^{\text{new}}(\Gamma_0(4),6,1)$ is one-dimensional and generated by the Hecke eigenform
\[g=q-12q^3+54q^5 - 88q^7 - 99q^9 + 540q^{11} - 418q^{13} + O(q^{15}).\]
Therefore, by Theorem 1.2, we see that the $\mathcal S(\text{SL}_2(\mathbb Z),6,\chi)$ is one-dimensional and generated by 
\[f=q^{1/2}-12q^{3/2}+54q^{5/2} - 88q^{7/2} - 99q^{9/2} + 540q^{11/2} - 418q^{13/2} + O(q^{15/2}).\]
\end{Exa}

\vskip 0.5 cm

\addcontentsline{toc}{chapter}{Bibliography}
\bibliographystyle{plain}
\bibliography{paper}

\end{document}